\documentclass[10pt]{amsart}%
\usepackage{graphicx}
\usepackage{amscd, color}
\usepackage{amsmath}
\usepackage{amsfonts}
\usepackage{amssymb}%
\providecommand{\U}[1]{\protect\rule{.1in}{.1in}}
\providecommand{\U}[1]{\protect\rule{.1in}{.1in}}
\providecommand{\U}[1]{\protect\rule{.1in}{.1in}} \textwidth 16.3cm
\theoremstyle{plain}

\newtheorem{theorem}{Theorem}[section]

\numberwithin{equation}{section}

\begin{document}
\title{A general Extrapolation Theorem for absolutely summing operators}

\author{D. Pellegrino \and J. Santos \and J. B. Seoane-Sep\'ulveda}

\address{Departamento de Matem\'{a}tica, \newline\indent Universidade Federal da Para\'{\i}ba, \newline\indent 58.051-900 - Jo\~{a}o Pessoa, Brazil.}\email{pellegrino@pq.cnpq.br and dmpellegrino@gmail.com}

\address{Departamento de Matem\'{a}tica, \newline\indent Universidade Federal de Sergipe, \newline\indent 49.500-000 -
Itabaiana, Brazil.}
\email{joedsonsr@yahoo.com.br}

\address{Departamento de An\'{a}lisis Matem\'{a}tico,\newline\indent Facultad de Ciencias Matem\'{a}ticas, \newline\indent Plaza de Ciencias 3, \newline\indent Universidad Complutense de Madrid,\newline\indent Madrid, 28040, Spain.}
\email{jseoane@mat.ucm.es}

\thanks{D. Pellegrino was supported by CNPq Grant 301237/2009-3. J. B. Seoane-Sep\'{u}lveda was supported by the Spanish Ministry of Science and Innovation, grant MTM2009-07848.}

\begin{abstract}
In this note we prove a general version of the Extrapolation Theorem, extending the classical linear extrapolation theorem due to B. Maurey. Our result shows, in particular, that the operators involved do not need to be linear.
\end{abstract}
\maketitle

\section{Introduction and background}

The notion of absolutely $(p;q)$-summing linear operators is due to A. Pietsch
\cite{stu} and B. Mitiagin and A. Pe\l czy\'{n}ski \cite{MPel}, inspired by
previous works of A. Grothendieck. The nonlinear theory of absolutely summing
operators was initiated by A. Pietsch and a complete nonlinear approach was
introduced by M.C. Matos \cite{Nach}.

Let $X,Y$ be Banach spaces over a fixed scalar field $\mathbb{K}=\mathbb{R}$
or $\mathbb{C}$; for $1\leq p<\infty$, denote by $\ell_{p,w}\left(  X\right)
$ the space of all sequences $\left(  x_{j}\right)  _{j=1}^{\infty}$ in $X$
such that $\left(  \varphi\left(  x_{j}\right)  \right)  _{j=1}^{\infty}%
\in\ell_{p}$ for every $\varphi$ in the topological dual of $X$ (represented
by $X^{\ast}$). Let $\mathcal{L}$ be the class of all continuous linear
operators between Banach spaces over $\mathbb{K}$, i.e., $\mathcal{L}%
=\bigcup_{X,Y}\mathcal{L}\left(  X;Y\right)  $, where $X$,$Y$ run over all
Banach spaces over $\mathbb{K}$.

A continuous linear operator $u:X\rightarrow Y$ is absolutely $p$-summing (we
write $u\in$ $\Pi_{p}(X;Y)$) if $\left(  u(x_{j})\right)  _{j=1}^{\infty}%
\in\ell_{p}\left(  Y\right)  $ whenever $\left(  x_{j}\right)  _{j=1}^{\infty
}\in\ell_{p,w}\left(  X\right)  .$ For details we refer to the classical
monograph \cite{djt} and to \cite{PellZ, dies} for more recent results.

An important result in the theory of absolutely summing linear operators is
the \textit{Extrapolation Theorem} due to B. Maurey \cite{Maurey}:

\begin{theorem}
[\textbf{Extrapolation Theorem}]\label{vv}Let $1<r<p<\infty$ and let $X$ be a
Banach space. If
\[
\Pi_{p}(X;\ell_{p})=\Pi_{r}(X;\ell_{p}),
\]
then, for any Banach space $Y,$
\[
\Pi_{p}(X;Y)=\Pi_{1}(X;Y).
\]

\end{theorem}

In the present paper we show that this Extrapolation Theorem holds in a much
more general form, where the linearity of the operators is not needed. For
example, as a very particular case of our main result we generalize the
Extrapolation Theorem in the following way:

Let $\mathbb{K}$ be a fixed scalar field $\mathbb{R}$ or $\mathbb{C}$, $1\leq
p<\infty$ and $\mathcal{F}$ a non-void family of maps, $\mathcal{F}%
\subset\bigcup_{X,Y}X^{Y}$, where $X$ and $Y$ run over all Banach spaces over
$\mathbb{K}$. We say that $f\in\mathcal{F}$ is absolutely $p$-summing if there
is a $C\geq0$ so that
\begin{equation}
\left(  {\textstyle\sum\limits_{j=1}^{m}}\left\Vert f(x_{j})\right\Vert
^{p}\right)  ^{1/p}\leq C\left(  \sup_{\varphi\in B_{E^{\ast}}}{\textstyle\sum
\limits_{j=1}^{m}}\left\vert \varphi\left(  x_{j}\right)  \right\vert
^{p}\right)  ^{1/p} \label{inf}%
\end{equation}
for every positive integer $m$ and every $x_{1},...,x_{m}\in X$. Let
\[
\Pi_{\mathcal{F},p}(X;Y)=\left\{  f\in\mathcal{F}:\text{ }f:X\rightarrow
Y\text{ is absolutely }p\text{-summing}\right\}  .
\]
Note that $\Pi_{\mathcal{F},p}(X;Y)$ is a vector space and the infimum of the
$C$'s satisfying (\ref{inf}) is a norm for $\Pi_{\mathcal{F},p}(X;Y)$ denoted
by $\pi_{\mathcal{F},p}.$ If $\left(  \Pi_{\mathcal{F},p}(X;Y),\pi
_{\mathcal{F},p}\right)  $ is complete and $\mathcal{F}$ is so that $T\circ
f\in\mathcal{F}$ whenever $T\in\mathcal{L}(Y;Z)$ and $f\in\Pi_{\mathcal{F}%
,p}(X;Y)$, then, as a consequence of our main result, we have the following theorem:

\begin{theorem}
[Nonlinear Extrapolation Theorem]\label{intr}Let $1<r<p<\infty$ and let $X$ be
a Banach space. If
\[
\Pi_{\mathcal{F},p}(X;\ell_{p})=\Pi_{\mathcal{F},r}(X;\ell_{p}),
\]
then, for any Banach space $Y,$
\[
\Pi_{\mathcal{F},p}(X;Y)=\Pi_{\mathcal{F},1}(X;Y).
\]
In particular, when $\mathcal{F}=\mathcal{L}$ we recover Maurey's
Extrapolation Theorem.
\end{theorem}

Our general result (Theorem \ref{et}) also furnishes, as a simple particular
case, a recent Extrapolation Theorem for Lipschitz summing mappings, due to D.
Chen and B. Zheng \cite{RL}.

One of the pillars of the theory of absolutely summing operators is the famous
Pietsch Domination Theorem which asserts that a continuous linear operator
$u:X\rightarrow Y$ is absolutely $p$-summing if, and only if, there is a
regular probability measure $\mu$ in the Borel sets of $B_{X^{\ast}}$ (with
the weak-star topology) and a constant $C\geq0$ such that%
\[
\left\Vert u\left(  x\right)  \right\Vert \leq C\left(
{\textstyle\int\nolimits_{B_{X^{\ast}}}}
\left\vert \varphi(x)\right\vert ^{p}d\mu\left(  \varphi\right)  \right)
^{1/p}%
\]
for all $x\in X$.

Very recently, a series of works (\cite{jfa, BPRn, joed, psjmaa, adv}) on
Pietsch Domination-Factorization Theorem have shown that Pietsch Domination
Theorem in fact needs almost no linear structure and a quite general version
is valid (see Theorem \ref{tu} below) which has shown to be very useful in
different contexts (\cite{achour, CDo}).

Let $X$, $Y$ and $E$ be (arbitrary) non-void sets, $\mathcal{H}\left(
X;Y\right)  $ be a non-void family of mappings from $X$ to $Y$, $G$ be a
Banach space and $K$ be a compact Hausdorff topological space. Let
\[
R\colon K\times E\times G\longrightarrow\lbrack0,\infty)~\text{and}%
\mathrm{~}S\colon\mathcal{H}\left(  X;Y\right)  \times E\times
G\longrightarrow\lbrack0,\infty)
\]
be arbitrary mappings and $1\leq t<\infty$. According to \cite{BPRn, psjmaa} a
mapping $f\in\mathcal{H}\left(  X;Y\right)  $ is $RS$-abstract $t$-summing if
there is a constant $C\geq0$ so that%
\begin{equation}
\left(  \sum_{j=1}^{m}S(f,x_{j},b_{j})^{t}\right)  ^{\frac{1}{t}}\leq
C\sup_{\varphi\in K}\left(  \sum_{j=1}^{m}R\left(  \varphi,x_{j},b_{j}\right)
^{t}\right)  ^{\frac{1}{t}},\label{33M}%
\end{equation}
for all $x_{1},\ldots,x_{m}\in E,$ $b_{1},\ldots,b_{m}\in G$ and
$m\in\mathbb{N}$. We define
\[
{\mathcal{H}}_{RS,t}(X;Y)=\left\{  f\in\mathcal{H}\left(  X;Y\right)  :\text{
}f\text{ is }RS\text{-abstract }t\text{-summing}\right\}  .
\]

Suppose that $R$ is so that the mapping
\begin{equation}
R_{x,b}\colon K\longrightarrow\lbrack0,\infty)~\text{defined by}%
~R_{x,b}(\varphi)=R(\varphi,x,b) \label{ll9}%
\end{equation}
is continuous for every $x\in E$ and $b\in G$. The Pietsch Domination Theorem
from \cite{psjmaa} reads as follows:

\begin{theorem}
\label{tu} Suppose that $S$ is arbitrary, $R$ satisfies \textbf{(}%
\ref{ll9}\textbf{)} and let $1\leq p<\infty$. A map $h\in\mathcal{H}\left(
X;Y\right)  $ is $RS$-abstract $p$-summing if and only if there is a constant
$C\geq0$ and a Borel probability measure $\mu$ on $K$ such that%
\begin{equation}
S(h,e,b)\leq C\left\Vert R\left(  \cdot,e,b\right)  \right\Vert _{L_{p}%
(K,\mu)} \label{yr}%
\end{equation}
for all $e\in E$ and $b\in G$.\newline
\end{theorem}

This general approach recovers several Pietsch Domination type theorems (see
\cite{BPRn}) and also rapidly found applications in different contexts (see
\cite{achour, CDo}). It is worth mentioning that the recent interesting
version of the Pietsch Domination Theorem for Lipschitz $(p;q;r)$-summing
operators proved in \cite[Theorem 5.4 (a)$\Rightarrow$(b)]{cha} can also be
obtained as a simple application of the general result from \cite{adv}.

In the next section we present the main result of this note, the general
Extrapolation Theorem (in the lines of the abstract setting of Theorem
\ref{tu}) which, as the general Pietsch Domination Theorem, does not need any
linear setting. In the final section we show how the general Extrapolation
Theorem can be applied.

\section{A new abstract setting}

In this section we build the environment needed for the proof of the general
Extrapolation Theorem. We keep the notation from the previous section. Also,
let $p,r$ be such that $1<r<p<\infty$. Suppose that $X$ is a topological
space, $E=X\times X$ and $K$ is a compact Hausdorff space such that $X$ is
continuously embedded in $C(K).$ We denote by $\mathcal{P}(K)$ the collection
of all regular Borel probability measures on $K.$ For any $\mu\in
\mathcal{P}(K)$, let $j_{\mu}:X\rightarrow L_{p}(\mu):=L_{p}(K,\mu)$ denote
the composition of the inclusion $X\rightarrow C(K)$ with the canonical map
$C(K)\rightarrow L_{p}(\mu).$ We will keep the notation, terminology and
assumptions from above. Also, suppose that

\textbf{(1)} For all $t\in\left\{  1,r,p\right\}  $ and all Banach spaces $Y$,
${\mathcal{H}}_{RS,t}(X;Y)$ is a vector space and the infimum of the $C$'s
satisfying (\ref{33M}) is a (complete) norm for ${\mathcal{H}}_{RS,t}(X;Y),$
denoted by $\pi_{_{RS,t}}\left(  \cdot\right)  .$

\textbf{(2)} For all Banach spaces $Y,Z,$ if $h\in{\mathcal{H}}(X;Y)$ and
$T:Y\rightarrow Z$ is a bounded linear operator, then%
\[
S(T\circ h,\left(  x,q\right)  ,b)\leq\left\Vert T\right\Vert S(h,\left(
x,q\right)  ,b).
\]

\textbf{(3)} $j_{\mu}\in\mathcal{H}\left(  X;L_{p}(\mu)\right)  $ and
\[
S(j_{\mu},\left(  x,q\right)  ,b)=\left\Vert R\left(  \cdot,\left(
x,q\right)  ,b\right)  \right\Vert _{L_{p}(K,\mu)}%
\]
\ for all $\left(  \left(  x,q\right)  ,b\right)  \in E\times G.$

\textbf{(4) }If $\{j_{\mu}(x_{1}),...,j_{\mu}(x_{m}),j_{\mu}(q_{1}),...j_{\mu
}(q_{m})\}$ is contained in a finite-dimensional subspace $F$ of $L_{p}(\mu)$
and $p_{F}:L_{p}(\mu)\rightarrow F$ denotes the canonical projection, then%
\begin{equation}
\sum_{j=1}^{m}S(j_{\mu},\left(  x_{j},q_{j}\right)  ,b_{j})^{r}=\sum_{j=1}%
^{m}S(p_{F}\circ j_{\mu},\left(  x_{j},q_{j}\right)  ,b_{j})^{r}.
\label{ee222}%
\end{equation}
Also $p_{F}\circ j_{\mu}\in{\mathcal{H}}_{RS,p}\left(  X;F\right)  $ and there
is a $C_{1}>0$ (not depending on $F$) so that
\begin{equation}
\pi_{_{RS,p}}(p_{F}\circ j_{\mu})\leq C_{1}\pi_{_{RS,p}}(j_{\mu}). \label{ee2}%
\end{equation}

\textbf{(5)} If ${\mathcal{H}}_{RS,p}(X;\ell_{p})={\mathcal{H}}_{RS,r}%
(X;\ell_{p})$ and $\pi_{{\mathcal{H}}_{RS,r}}(v)\leq c\cdot\pi_{{\mathcal{H}%
}_{RS,p}}(v)$ for all $v\in{\mathcal{H}}_{RS,p}(X;\ell_{p})$, then
\[
{\mathcal{H}}_{RS,p}(X;\ell_{p}^{m})={\mathcal{H}}_{RS,r}(X;\ell_{p}^{m})
\]
and%
\[
\pi_{{\mathcal{H}}_{RS,r}}(v)\leq c\cdot\pi_{{\mathcal{H}}_{RS,p}}(v)
\]
for all $v\in{\mathcal{H}}_{RS,p}(X;\ell_{p}^{m})$ and all positive integers
$m.$

It may seem that we have too many hypotheses but we recall that we are working
in a very abstract setting and for this reason it is natural to need to
introduce some hypotheses. Moreover, a careful examination shows that the
above hypotheses are quite natural.

\section{The general Extrapolation Theorem}

Following the assumptions and terminology of the previous sections we can
state and prove our main result:

\begin{theorem}
\label{et}Let $1<r<p<\infty.$ If
\[
{\mathcal{H}}_{RS,p}(X;\ell_{p})={\mathcal{H}}_{RS,r}(X;\ell_{p})
\]
then, for any Banach space $Y,$%
\[
{\mathcal{H}}_{RS,p}(X;Y)={\mathcal{H}}_{RS,1}(X;Y).
\]

\end{theorem}

\begin{proof}
From Theorem \ref{tu}, using the monotonicity of the $L_{p}$ norms we have
\[
{\mathcal{H}}_{RS,1}(X;Y)\subset{\mathcal{H}}_{RS,p}(X;Y).
\]

For the converse inclusion, it suffices to prove that, regardless of the
Banach space $Y,$ there is $C>0$ such that for each $h\in{\mathcal{H}}%
_{RS,p}(X;Y),$ we have
\[
\pi_{_{RS,1}}(h)\leq C\cdot\pi_{_{RS,p}}(h).
\]

Since ${\mathcal{H}}_{RS,p}(X;\ell_{p})={\mathcal{H}}_{RS,r}(X;\ell_{p}),$ it
follows from the Open Mapping Theorem and (1) that there is $c>0$ such that%
\begin{equation}
\pi_{{\mathcal{H}}_{RS,r}}(v)\leq c\cdot\pi_{{\mathcal{H}}_{RS,p}}(v),\text{
}\label{gf}%
\end{equation}
for all $v\in{\mathcal{H}}_{RS,p}(X;\ell_{p}).$ Let $\mu\in\mathcal{P}(K)$;
since $L_{p}\left(  K,\mu\right)  $ is an ${\mathcal{L}}_{p,\lambda}$ space
for all $\lambda>1$ (see \cite[Theorem 3.2]{djt}), we can assert that for each
$\left(  x_{i}\right)  _{1\leq i\leq m},$ $\left(  q_{i}\right)  _{1\leq i\leq
m}$ in $X$, the subspace of $L_{p}\left(  K,\mu\right)  $ generated by
$\left\{  j_{\mu}\left(  x_{1}\right)  ,...,j_{\mu}\left(  x_{m}\right)
,j_{\mu}\left(  q_{1}\right)  ,...,j_{\mu}\left(  q_{m}\right)  \right\}  $
embeds $\lambda$-isomorphically into $\ell_{p}.$ More precisely,
\[
\text{span}\left\{  j_{\mu}\left(  x_{1}\right)  ,...,j_{\mu}\left(
x_{m}\right)  ,j_{\mu}\left(  q_{1}\right)  ,...,j_{\mu}\left(  q_{m}\right)
\right\}
\]
is contained in a subspace $F$ of $L_{p}\left(  K,\mu\right)  $ for which
there is an isomorphism $T:F\rightarrow\ell_{p}^{\dim F}$ with
\begin{equation}
\left\Vert T\right\Vert \left\Vert T^{-1}\right\Vert <\lambda.\label{ee1}%
\end{equation}

From (3) $j_{\mu}$ is $RS$-abstract $p$-summing and
\begin{equation}
\pi_{_{RS,p}}(j_{\mu})\leq1. \label{umum}%
\end{equation}
Hence, from (\ref{umum}) and (\ref{ee2}) we have
\begin{equation}
\pi_{_{RS,p}}(p_{F}\circ j_{\mu})\leq C_{1}\text{ } \label{kkkj}%
\end{equation}
and, from (2), it follows that $T\circ p_{F}\circ j_{\mu}\in{\mathcal{H}%
}_{RS,p}(X;\ell_{p}^{\dim F}).$ Finally, from (5) and (\ref{gf}) we can assert
that
\begin{equation}
T\circ p_{F}\circ j_{\mu}\in{\mathcal{H}}_{RS,r}(X;\ell_{p}^{\dim F}).
\label{umww}%
\end{equation}
and
\begin{equation}
\pi_{{\mathcal{H}}_{RS,r}}(T\circ p_{F}\circ j_{\mu})\leq c\cdot
\pi_{{\mathcal{H}}_{RS,p}}(T\circ p_{F}\circ j_{\mu}). \label{doisww}%
\end{equation}

We thus have
\begin{align*}
&  \left(  \sum_{j=1}^{m}S(j_{\mu},\left(  x_{j},q_{j}\right)  ,b_{j}%
)^{r}\right)  ^{1/r}\overset{(\ref{ee222})}{=}\left(  \sum_{j=1}^{m}%
S(T^{-1}\circ T\circ p_{F}\circ j_{\mu},\left(  x_{j},q_{j}\right)
,b_{j})^{r}\right)  ^{1/r}\\
&  \overset{(2)}{\leq}\left\Vert T^{-1}\right\Vert \left(  \sum_{j=1}%
^{m}S(T\circ p_{F}\circ j_{\mu},\left(  x_{j},q_{j}\right)  ,b_{j}%
)^{r}\right)  ^{1/r}\\
&  \overset{(\ref{umww})}{\leq}\left\Vert T^{-1}\right\Vert \pi_{{\mathcal{H}%
}_{RS,r}}(T\circ p_{F}\circ j_{\mu})\sup_{\varphi\in K}\left(  \sum_{j=1}%
^{m}R\left(  \varphi,\left(  x_{j},q_{j}\right)  ,b_{j}\right)  ^{r}\right)
^{1/r}\\
&  \overset{(\ref{doisww})}{\leq}c\cdot\left\Vert T^{-1}\right\Vert
\pi_{{\mathcal{H}}_{RS,p}}(T\circ p_{F}\circ j_{\mu})\sup_{\varphi\in
K}\left(  \sum_{j=1}^{m}R\left(  \varphi,\left(  x_{j},q_{j}\right)
,b_{j}\right)  ^{r}\right)  ^{1/r}\\
&  \overset{(2)}{\leq}c\cdot\left\Vert T^{-1}\right\Vert \left\Vert
T\right\Vert \pi_{{\mathcal{H}}_{RS,p}}(p_{F}\circ j_{\mu})\sup_{\varphi\in
K}\left(  \sum_{j=1}^{m}R\left(  \varphi,\left(  x_{j},q_{j}\right)
,b_{j}\right)  ^{r}\right)  ^{1/r}\\
&  \overset{(\ref{ee1})\text{ and }(\text{\ref{kkkj}})}{\leq}c\cdot
\lambda\cdot C_{1}\sup_{\varphi\in K}\left(  \sum_{j=1}^{m}R\left(
\varphi,\left(  x_{j},q_{j}\right)  ,b_{j}\right)  ^{r}\right)  ^{1/r}.
\end{align*}
So
\[
\pi_{{\mathcal{H}}_{RS,r}}(j_{\mu})\leq c\cdot C_{1}\cdot\lambda
\]
and, from Theorem \ref{tu}, there is $\hat{\mu}\in\mathcal{P}(K)$
\begin{equation}
S(j_{\mu},\left(  x,q\right)  ,b)\leq c\cdot C_{1}\cdot\lambda\left(
\int\limits_{K}R\left(  \varphi,\left(  x,q\right)  ,b\right)  ^{r}d\hat{\mu
}\left(  \varphi\right)  \right)  ^{1/r}, \label{mju}%
\end{equation}
for all $\left(  \left(  x,q\right)  ,b\right)  \in E\times G.$ Thus,%
\begin{equation}
\left\Vert R\left(  \cdot,\left(  x,q\right)  ,b\right)  \right\Vert
_{L_{p}(K,\mu)}\overset{(3)}{=}S(j_{\mu},\left(  x,q\right)  ,b)\overset
{(\ref{mju})}{\leq}c\cdot C_{1}\cdot\lambda\left\Vert R\left(  \cdot,\left(
x,q\right)  ,b\right)  \right\Vert _{L_{r}(K,\hat{\mu})}, \label{mi}%
\end{equation}
for all $\left(  \left(  x,q\right)  ,b\right)  \in E\times G.$

Now let $h\in{\mathcal{H}}_{RS,p}(X;Y).$ From Theorem \ref{tu} there is a
$\mu_{0}\in\mathcal{P}(K)$ such that
\[
S(h,\left(  x,q\right)  ,b)\leq\pi_{_{RS,p}}(h)\left(  \int\limits_{K}R\left(
\varphi,\left(  x,q\right)  ,b\right)  ^{p}d\mu_{0}\left(  \varphi\right)
\right)  ^{1/p}=\pi_{_{RS,p}}(h)\left\Vert R\left(  \cdot,\left(  x,q\right)
,b\right)  \right\Vert _{L_{p}(K,\mu_{0})}%
\]
for all $\left(  \left(  x,q\right)  ,b\right)  \in E\times G.$ Now it is
enough to show that%
\[
\left\Vert R\left(  \cdot,\left(  x,q\right)  ,b\right)  \right\Vert
_{L_{p}(K,\mu_{0})}\leq C\left\Vert R\left(  \cdot,\left(  x,q\right)
,b\right)  \right\Vert _{L_{1}(K,\bar{\mu})}%
\]
for some $\bar{\mu}\in\mathcal{P}(K)$ and some constant $C$ depending only on
$X$; the rest of proof follows the lines of Maurey's original argument.
Starting with $\mu_{0},$ define $\left(  \mu_{n}\right)  _{n=0}^{\infty}$ in
$\mathcal{P}(K)$ by setting $\mu_{n+1}=\hat{\mu}_{n},$ $n=0,1,....$ and
define
\[
\bar{\mu}=\sum_{n=0}^{\infty}2^{-n-1}\mu_{n}.
\]
Then $\bar{\mu}\in\mathcal{P}(K)$ and, since $1<r<p,$ there exists a
$\theta\in\left(  0,1\right)  $ so that%
\[
\frac{1}{r}=\theta+\frac{1-\theta}{p}.
\]
Using Littlewood's Inequality (see \cite[p. 55]{Garling}), we get%
\begin{align}
\left\Vert R\left(  \cdot,\left(  x,q\right)  ,b\right)  \right\Vert
_{L_{r}(K,\mu_{n})} &  =\left(  \int\limits_{K}R\left(  \varphi,\left(
x,q\right)  ,b\right)  ^{r}d\mu_{n}\left(  \varphi\right)  \right)
^{1/r}\label{dh}\\
&  \leq\left(  \int\limits_{K}R\left(  \varphi,\left(  x,q\right)  ,b\right)
d\mu_{n}\left(  \varphi\right)  \right)  ^{\theta}\left(  \int\limits_{K}%
R\left(  \varphi,\left(  x,q\right)  ,b\right)  ^{p}d\mu_{n}\left(
\varphi\right)  \right)  ^{\frac{1-\theta}{p}}\nonumber\\
&  =\left\Vert R\left(  \cdot,\left(  x,q\right)  ,b\right)  \right\Vert
_{L_{1}(K,\mu_{n})}^{\theta}\left\Vert R\left(  \cdot,\left(  x,q\right)
,b\right)  \right\Vert _{L_{p}(K,\mu_{n})}^{1-\theta}.\nonumber
\end{align}
Then,%
\begin{align*}
&  \sum_{n=0}^{\infty}2^{-n-1}\left\Vert R\left(  \cdot,\left(  x,q\right)
,b\right)  \right\Vert _{L_{p}(K,\mu_{n})}\overset{(\ref{mi})}{\leq}c\cdot
C_{1}\cdot\lambda\sum_{n=0}^{\infty}2^{-n-1}\left\Vert R\left(  \cdot,\left(
x,q\right)  ,b\right)  \right\Vert _{L_{r}(K,\hat{\mu}_{n})}\\
&  =c\cdot C_{1}\cdot\lambda\sum_{n=0}^{\infty}2^{-n-1}\left\Vert R\left(
\cdot,\left(  x,q\right)  ,b\right)  \right\Vert _{L_{r}(K,\mu_{n+1})}\\
&  \overset{(\ref{dh})}{\leq}c\cdot C_{1}\cdot\lambda\sum_{n=0}^{\infty
}2^{-n-1}\left\Vert R\left(  \cdot,\left(  x,q\right)  ,b\right)  \right\Vert
_{L_{1}(K,\mu_{n+1})}^{\theta}\left\Vert R\left(  \cdot,\left(  x,q\right)
,b\right)  \right\Vert _{L_{p}(K,\mu_{n+1})}^{1-\theta}\\
&  =c\cdot C_{1}\cdot\lambda\sum_{n=0}^{\infty}\left(  2^{-n-1}\right)
^{\theta}\left\Vert R\left(  \cdot,\left(  x,q\right)  ,b\right)  \right\Vert
_{L_{1}(K,\mu_{n+1})}^{\theta}\left(  2^{-n-1}\right)  ^{1-\theta}\left\Vert
R\left(  \cdot,\left(  x,q\right)  ,b\right)  \right\Vert _{L_{p}(K,\mu
_{n+1})}^{1-\theta}\\
&  =c\cdot C_{1}\cdot\lambda\sum_{n=0}^{\infty}\left(  2^{-n-1}\left\Vert
R\left(  \cdot,\left(  x,q\right)  ,b\right)  \right\Vert _{L_{1}(K,\mu
_{n+1})}\right)  ^{\theta}\left(  2^{-n-1}\left\Vert R\left(  \cdot,\left(
x,q\right)  ,b\right)  \right\Vert _{L_{p}(K,\mu_{n+1})}\right)  ^{1-\theta}\\
&  \overset{\text{(*)}}{\leq}c\cdot C_{1}\cdot\lambda\left(  \sum
_{n=0}^{\infty}2^{-n-1}\left\Vert R\left(  \cdot,\left(  x,q\right)
,b\right)  \right\Vert _{L_{1}(K,\mu_{n+1})}\right)  ^{\theta}\left(
\sum_{n=0}^{\infty}2^{-n-1}\left\Vert R\left(  \cdot,\left(  x,q\right)
,b\right)  \right\Vert _{L_{p}(K,\mu_{n+1})}\right)  ^{1-\theta}\\
&  \leq c\cdot C_{1}\cdot\lambda\left(  \sum_{n=0}^{\infty}2^{-n-1}\left\Vert
R\left(  \cdot,\left(  x,q\right)  ,b\right)  \right\Vert _{L_{1}(K,\mu
_{n+1})}\right)  ^{\theta}\left(  2\sum_{n=0}^{\infty}2^{-n-1}\left\Vert
R\left(  \cdot,\left(  x,q\right)  ,b\right)  \right\Vert _{L_{p}(K,\mu_{n}%
)}\right)  ^{1-\theta},
\end{align*}
where in (*) we used Holder's Inequality. Hence
\begin{align*}
\sum_{n=0}^{\infty}2^{-n-1}\left\Vert R\left(  \cdot,\left(  x,q\right)
,b\right)  \right\Vert _{L_{p}(K,\mu_{n})} &  \leq\left(  c\cdot C_{1}%
\cdot\lambda\right)  ^{1/\theta}2^{\frac{1-\theta}{\theta}}\left(  \sum
_{n=0}^{\infty}2^{-n-1}\left\Vert R\left(  \cdot,\left(  x,q\right)
,b\right)  \right\Vert _{L_{1}(K,\mu_{n+1})}\right)  \\
&  \leq\left(  c\cdot C_{1}\cdot\lambda\right)  ^{1/\theta}2^{\frac{1-\theta
}{\theta}}2\left(  \sum_{n=0}^{\infty}2^{-n-1}\left\Vert R\left(
\cdot,\left(  x,q\right)  ,b\right)  \right\Vert _{L_{1}(K,\mu_{n})}\right)  .
\end{align*}
Note that
\begin{align*}
\left\Vert R\left(  \cdot,\left(  x,q\right)  ,b\right)  \right\Vert
_{L_{1}(K,\bar{\mu})} &  =\int\limits_{K}R\left(  \varphi,\left(  x,q\right)
,b\right)  d\left[  \sum_{n=0}^{\infty}2^{-n-1}\mu_{n}\right]  \left(
\varphi\right)  \\
&  =\sum_{n=0}^{\infty}2^{-n-1}\int\limits_{K}R\left(  \varphi,\left(
x,q\right)  ,b\right)  d\mu_{n}\left(  \varphi\right)  \\
&  =\sum_{n=0}^{\infty}2^{-n-1}\left\Vert R\left(  \cdot,\left(  x,q\right)
,b\right)  \right\Vert _{L_{1}(K,\mu_{n})}.
\end{align*}
So,
\[
\sum_{n=0}^{\infty}2^{-n-1}\left\Vert R\left(  \cdot,\left(  x,q\right)
,b\right)  \right\Vert _{L_{p}(K,\mu_{n})}\leq\left(  2c\cdot C_{1}%
\cdot\lambda\right)  ^{1/\theta}\left\Vert R\left(  \cdot,\left(  x,q\right)
,b\right)  \right\Vert _{L_{1}(K,\bar{\mu})}.
\]
In particular,
\[
2^{-1}\left\Vert R\left(  \cdot,\left(  x,q\right)  ,b\right)  \right\Vert
_{L_{p}(K,\mu_{0})}\leq\left(  2c\cdot C_{1}\cdot\lambda\right)  ^{1/\theta
}\left\Vert R\left(  \cdot,\left(  x,q\right)  ,b\right)  \right\Vert
_{L_{1}(K,\bar{\mu})}.
\]
Thus $C=2\left(  2c\cdot C_{1}\cdot\lambda\right)
^{1/\theta}$ is the desired constant.
\end{proof}

\section{Recovering the previous Extrapolation Theorems}

\subsection{The Extrapolation Theorem for absolutely $p$-summing linear
operators}

Note that a continuous linear operator $T:X\rightarrow Y$ is absolutely
$p$-summing if and only if it is $RS$-abstract $p$-summing with
\[
E=X\times X\text{ and }G=\mathbb{R}%
\]
and $K=B_{X^{\ast}}$, with the weak star topology, ${\mathcal{H}%
}(X;Y)={\mathcal{L}}(X;Y)$ and $R$ and $S$ are defined by:
\[
R\colon B_{X^{\ast}}\times(X\times X)\times\mathbb{R}\longrightarrow
\lbrack0,\infty)~,~R(\varphi,(x_{1},x_{2}),\lambda)=|\varphi(x_{1})|
\]%
\[
S\colon{\mathcal{L}}(X;Y)\times(X\times X)\times\mathbb{R}\longrightarrow
\lbrack0,\infty)~,~S(T,(x_{1},x_{2}),\lambda)=\left\Vert T(x_{1})\right\Vert
.
\]
Since the hypotheses of Theorem \ref{et} are straightforwardly satisfied, we
recover the classical Extrapolation Theorem (Theorem \ref{vv}).

\subsection{The Extrapolation Theorem for Lipschitz $p$-summing maps}

If $X=(X,d_{X})$ and $Y=(Y,d_{Y})$ are metric spaces, according to Farmer and
Johnson \cite{FaJo}, a map $T\colon X\longrightarrow Y$ is Lipschitz
$p$-summing (notation $T\in\Pi_{p}^{L}(X;Y)$) if there is a constant $C\geq0$
such that, for all natural $n$ and $x_{1},\ldots,x_{n},y_{1},\ldots,y_{n}\in
X$,
\[
\sum_{i=1}^{n}d_{Y}(T(x_{i}),T(y_{i}))^{p}\leq C^{p}\sup_{f\in B_{X^{\#}}}%
\sum_{i=1}^{n}|f(x_{i})-f(y_{i})|^{p},
\]
where $B_{X^{\#}}$ is the unit ball of the Lipschitz dual $X^{\#}$ of $X$. The infimum of all such $C$ is denoted by $\pi
_{p}^{L}$. Note that $T$ is Lipschitz $p$-summing if and only if it is $RS$-abstract
$p$-summing with
\[
E=X\times X\text{ and }G=\mathbb{R}%
\]
and $K=B_{X^{\#}}$, which is a compact Hausdorff space in the topology of
pointwise convergence on $Y$, $\mathcal{H}\left(  X;Y\right)  $ is the set of
all maps from $X$ to $Y$ and $R$ and $S$ are defined by:
\[
R\colon B_{X^{\#}}\times(X\times X)\times\mathbb{R}\longrightarrow
\lbrack0,\infty)~,~R(f,(x,y),\lambda)=|f(x)-f(y)|
\]%
\[
S\colon{\mathcal{H}}\left(  X;Y\right)  \times(X\times X)\times\mathbb{R}%
\longrightarrow\lbrack0,\infty)~,~S(T,(x,y),\lambda)=d_{Y}(T(x),T(y)).
\]
As a consequence of Theorem \ref{tu} we have:

\begin{theorem}
[Farmer-Johnson]The following are equivalent for a mapping $T\colon
X\longrightarrow Y$ between metric spaces:

(i) $T$ is Lipschitz $p$-summing.

(ii) There is a probability $\mu$ on $B_{X^{\#}}$ and a constant $C\geq0$ such
that%
\[
d_{Y}\left(  Tx,Ty\right)  \leq C\left(
{\textstyle\int\nolimits_{B_{X^{\#}}}}
|f(x)-f(y)|^{p}d\mu(f)\right)  ^{1/p}.
\]

\end{theorem}

It is well-known that if $Y$ is Banach space then $(\Pi_{p}^{L}(X;Y),\pi
_{p}^{L})$ is also a Banach space and $\Pi_{p}^{L}$ has the ideal property.
Moreover%
\begin{align*}
S(j_{\mu},(x,y),\lambda)  &  =\left\Vert j_{\mu}(x)-j_{\mu}(y)\right\Vert
_{L_{p}(B_{X^{\#}},\mu)}=\left(
{\textstyle\int\nolimits_{B_{X^{\#}}}}
|j_{\mu}(x)f-j_{\mu}(x)f|^{p}d\mu(f)\right)  ^{1/p}\\
&  =\left(
{\textstyle\int\nolimits_{B_{X^{\#}}}}
|f(x)-f(y)|^{p}d\mu(f)\right)  ^{1/p}=\left\Vert R(\cdot,(x,y),\lambda
)\right\Vert _{L_{p}(B_{X^{\#}},\mu)}.
\end{align*}
So, from \ Theorem \ref{et} we recover the Extrapolation Theorem due to D. Chen and B. Zheng  \cite{RL}:

\begin{theorem}
[Extrapolation Theorem for Lipschitz $p$-summing operators]Let $1<r<p<\infty$
and let $X$ be a metric space. If
\[
\Pi_{p}^{L}(X;\ell_{p})=\Pi_{r}^{L}(X;\ell_{p}),
\]
then, for any Banach space $Y,$%
\[
\Pi_{p}^{L}(X;Y)=\Pi_{1}^{L}(X;Y).
\]

\end{theorem}

A similar argument also shows that Theorem \ref{intr} holds, since conditions
(1)-(5) are easily satisfied.

\end{document}